\documentclass{article}

\usepackage{amsmath,amssymb,amsthm}
\usepackage{latexsym}
\usepackage{epsfig}
\usepackage{psfrag}

\newtheorem{ex}{EXAMPLE}
\newtheorem{thm}{THEOREM}
\newtheorem{lem}{LEMMA}
\newtheorem{cor}{COROLLARY}
\newtheorem{rem}{REMARK}
\newtheorem{definition}{Definition}

\author{Maciej Ciesielski and Grzegorz Lewicki}

\title
{Some remarks on contractive and existence sets } 

\begin{document}
	
\maketitle

\begin{abstract}
Let $X$ be a real or complex Banach space and let $ F \subset X $ be a non-empty set.
$F$ is called {\it an existence set of best
coapproximation} ({\it existence set} for brevity), if for any $ x
\in X $, $R_F(x) \neq \emptyset, $ where
$$
R_F (x) = \{ d \in F : \|d-c\| \leq \|x-c\| \hbox{ for any } c \in
F \} .
$$
It is clear that any existence set is a contractive subset of $X.$ The aim of this paper is to present some conditions on $F$ and $X$ under which the notions of exsistence set and contractive set are equivalent. 
\end{abstract}

{ \underline{Key Words and Phrases:}\hspace{0.15in} Banach spaces, reflexivity, strict convexity, contractive and existence sets, one complemented spaces.}

\section{Introduction} \vskip0.3cm \noindent
Let $ X $ be a real or complex Banach space and let $ F \subset X $ be a non-empty
set. A continuous mapping $ P: X \rightarrow F $ is called {\it a
projection onto $F,$} whenever $ P |_F = Id, $ that is $ P^2 = P.
$ Setting
$$
Min (F)= \{ z \in X: \hbox{ for every } c \in F, x \in X, \hbox{
if }\|z-c\|  \geq \| x-c\| \hbox{ then } x = z \} ,
$$
we say that $ F \subset X $ is {\it optimal} if $ Min(F) = F. $
Observe that for any $ F\subset X, $ $ F \subset Min(F) . $ This
notion has been introduced by Beauzamy and Maurey  in \cite{BM}, where
basic properties concerning optimal sets can be found.
\par
A set $ F \subset X $ is called {\it an existence set of best
coapproximation} ({\it existence set} for brevity), if for any $ x
\in X $, $R_F(x) \neq \emptyset, $ where
$$
R_F (x) = \{ d \in F : \|d-c\| \leq \|x-c\| \hbox{ for any } c \in
F \} .
$$
Notice that any contractive set is an existence set. Indeed, if $ P:X \rightarrow F$ is a contractive projection, then $Px \in R_F(x)$ for any $x \in X.$
Also it is clear that any existence set is an optimal set. The
converse, in general, is not true. However, the following result is satisfied.
\begin{thm} 
\label{optimal}
(\cite{BM}, Prop. 2)
If $X$ is one-complemented in $ X^{**}$ and strictly convex, then any
optimal subset of $X$ is an existence set in $X,$ which, in
particular, holds true for strictly convex spaces $X,$ such that $
X = Z^*$ for some Banach space $Z.$
\end{thm}
Existence and optimal sets have been studied by many authors from
different points of view, mainly in the context of approximation
theory (see e.g. \cite{B}, \cite{BE},  [\cite{DE} - \cite{H}], \cite{KL}, \cite{KL1}, \cite{L}, \cite{LTr}, [\cite{PS} - \cite{W}]).
\newline
Recall that a closed subspace $F$ of a Banach space $X$ is called
{\it one-complemented} if there exists a linear projection of norm
one from $X$ onto $F.$  It is obvious that any
one-complemented subspace is an existence set. The converse, in
general, is not true. By a deep result of Lindenstrauss, \cite{L}
there exists a Banach space $X$ and $F$ a linear subspace of $X,$
codim$(F) = 2,$ such that:

a).  $F$ is one-complemented in any containing it hyperplane $Y$
of $X;$

b). $F$ is not one-complemented in $X.$ 
\newline 
This gives you
immediately an example of a subspace being an existence set which
is not one-complemented. However, in \cite{BM} (see also \cite{LTr} p.121),
the the following result has been
proven.
\begin{thm} (see \cite{BM}, Prop. 5). 
\label{smooth1} Let $V$ be a
linear subspace of a smooth, reflexive and strictly convex Banach
space. If $V$ is an optimal set then $V$ is one-complemented in
$X.$ If $X$ is a smooth Banach space, then any subspace of $X$
which is an existence set is one-complemented. Moreover, in both
cases a norm-one projection from $X$ onto $V$ is uniquely
determined.
\end{thm}
The aim of this paper is to present some conditions on a Banach space $X$ and a convex and closed set $F \subset X$  under which the notions of existence set and contractive set are equivalent. In other words, we will study the problem of existence of 
a non-expansive selection $P:X \rightarrow F$ such that $Px \in R_F(x)$ for any $x \in X.$
\newline
In Section 1 some preliminary results will be demonstrated.
\newline
The main results of the paper will be presented in Section 2.
In Section 2 we assume that $F$ has a nonempty interior in $X.$ 
\newline
In Section 3 we consider the general case. 
\newline
All the results will be demonstrated for real Banach spaces. However, at the end of Section 3, we show how to apply the results obtained in the real case
to the case of complex Banach spaces (see Lemma \ref{complex1}, Lemma \ref{complex2} and Theorem \ref{complex3}). 
\newline
Now we present some notions  which will be used in this
paper. 
\newline 
In the sequel by $ S(X)$ we denote the unit sphere
in a Banach space X and by $ S(X^*)$ the unit sphere in its dual
space. A functional $f \in S(X^*)$ is called a {\it supporting
functional} for $ x \in X,$ if $ f(x)=\| x\|.$ Analogously, a
point $x \in S(X)$ is called a {\it norming point} for $ f \in
X^*$ if $ f(x) = \| f \| . $ 
\newline 
A point $x \in X$ is called
a {\it smooth point} if it has exactly one supporting functional.
A Banach space $X$ is called {\it smooth} if any $x \in S(X)$ is a
smooth point. 
\newline 
By $ext(X)$ we denote the set of all
extreme points of $S(X).$ A Banach space $X$ is called {\it
strictly convex} if $ ext(X) = S(X).$ 
\newline 
If $ F $ is a
linear subspace of a Banach space $X,$ by $ {\cal P}(X,F)$ we will
denote the set of all linear, continuous projections from $X$ onto
$F.$
If $X$ is a Banach space and $ Z \subset X$ is an affine subspace of $X,$ for $F \subset Z$ by $int_Z(F)$ ($int(F)$ if $Z=X$), we denote the interior
of $F$ with respect to $Z.$ 
\newline
To the end of this paper, unless otherwise stated, all Banach spaces are real.
Also we will need 
\begin{definition}
\label{fixed}
It is said that a Banach space $X$ has (CFPP) if and only if for any nonempty, convex, closed and bounded set $ F \subset X$ and any nonexpansive mapping $ T:X \rightarrow X$ such that $ T(F) \subset F,$ $T$ has a fixed point in $F.$  
\end{definition}
\vskip0.5cm  
\section*{Section 1}  
\vskip0.5cm
\noindent We start with a lemma presenting basic properties of existence sets.
\begin{lem}
\label{basicp}
Let $X$ be a Banach space and let $ F \subset X$ be an existence set. Then for any $ d \in X,$ $F+d$ is an existence set. Also, for any $ t \in \mathbb{R},$ $ t F $ is an existence set. If $F$ is a linear subspace then $F$ is an existence 
set if and only if for any $ x \in X \setminus F$ there exists $ P_x \in {\cal P}(F_x, F),$ $ \|P_x\|=1,$ where $ F_x = span[x] \oplus F.$   
\end{lem}
\begin{proof}
Fix $ d \in X$ and $ x \in X.$ Let $ p \in R_F (x-d). $ Let $ c \in F+d, $ $c= c_1 +d, $ wher $ c_1 \in F.$ Then 
$$
\| c- (p+d)\| = \|c_1 - p \| \leq \| x-d -c_1 \| = \| x -c\|,
$$
which shows that $p+d \in R_{F+d}(x).$ Now fix $ t \in \mathbb{R}.$ If $ t =0, $ then $ F = \{0\}$ is obviously an existence set. If $ t \neq 0,$ fix $ p\in R_F(x/t).$ Let $c \in t F, $ $ c = t c_1.$ 
Observe that 
$$
\|p - c_1 \| \leq \| p - x/t \|
$$
and consequently 
$$
|t| \|p - c_1 \| \leq |t| \| p - x/t \|,
$$
which implies that $\| t p - c\| \leq \| t p - x\|.$ Hence $ t p \in R_{t F}(x).$  Now assume that $F$ is a linear subspace and fix $ x \in X \setminus F.$ Then any $ y \in F_x$ can be represented in the unique way as 
$ y = t x +v_o,$ where $ v_o \in F.$ Fix $ p \in R_F(x). $ We show that $ t p + v_o  \in R_F(t x + v_o) .$ This is obvious if $t = 0.$ So assume that $ t \neq 0.$ Fix $ v \in F.$ Since $ \frac{v-v_o}{t} \in F,$ 
$$
\| p - \frac{v-v_o}{t}\| \leq \| x - \frac{v-v_o}{t}\|
$$
and consequently
$$
|t| \| p - \frac{v-v_o}{t}\| \leq |t| \| x - \frac{v-v_o}{t}\|,
$$
which imples that 
$ \| tp + v_o -v\| \leq \| tx +v_o -v \| = \|y -v\|.$ 
Notice that a linear mapping $ P_x : F_x \rightarrow F$ defined by $ P_x(tx+v_o) = tp+v_o  $ belongs to $ {\cal P}(F_x, F).$ By the above reasoning, for any $y \in F_x,$  $ \| P_x y  \| = \|P_x y -0\| \leq  \|y\|, $ which shows that $ \|P_x\|=1.$  
Conversely, if there exists $ P_x \in {\cal P}(F_x,F),$ $ \|Px \|=1, $ then for any $v \in F,$ and $ y \in F_x,$
$$
\| P_xy -v \| = \| P_x(y-v)\| \leq \|y-v\|
$$
which shows our claim.
\end{proof}
\begin{lem}
 \label{compact}
(compare with \cite{BR1}, Lemma 1). Let $ X$ be a Banach space and let $C \subset X$ be a locally compact and convex set. For $ F \subset C,$ $ F \neq \emptyset, $ define 
$$
N(F) = \{ f:C \rightarrow C: \| y - f(x) \| \leq \|y-x\| \hbox{ for any } y\in F, x \in C \}.
$$
Then $ N(F)$ is compact in the topology of pointwise-weak convergence.
\end{lem}
\begin{proof}
Fix $ x_o \in F.$ For $x \in X$ set $ C_x = \{ y \in C: \|y-x_o\|\leq \|x-x_o\| \}.$ Notice that for $ f \in N(F) $
$$ 
\| f(x) - x_o \| \leq \| x-x_o\|.
$$
Hence $ N(F) \subset P = \Pi_{x \in C} C_x.$ Since $ C $ is convex and locally weakly compact, by the Mazur Theorem,  $ C_x,$ as a bounded subset of $C,$ is weakly compact. 
By the Tychonoff Theorem, $P$ is a compact set in the topology of pointwise-weak convergence, which we denote by $ \tau.$ Hence to show that $N(F) $ is compact, we need to demonstrate that $N(F)$ is a $ \tau$ closed subset of $P.$
Let $ \{f_{\gamma} \} \subset N(F)$ be a net $\tau$-converging to $ f \in P.$ Hence for any $ y \in F, f_{\gamma}(y) \rightarrow f(y)$ in the weak topology. Since $ f_{\gamma}(y) = y, $ for any $ \gamma, $ $f(y)=y.$
Moreover, for $y \in F, x \in C, $ 
$$
\|y - f(x) \| \leq \liminf_{\gamma} \|y - f_{\gamma}(x)\| \leq \|y -x \|.
$$
Hence $ f \in N(F),$ which completes the proof. 
\end{proof}
\begin{lem}
\label{maximal}
(compare with \cite{BR1}, Lemma 2)
Let $ F,C,X$ and $ N(F)$ be as in Lemma \ref{compact}. Then there exists $ r \in N(F)$ such that for any $ f\in N(F),$ $x \in C$ and $ y \in F$
$$
\| y - f(r(x))\| = \|y - r(x)\|.
$$
\end{lem}
\begin{proof}
First we define a partial ordering in $ N(F).$ For any $f,g \in N(F) $ it is said that $ f < g$ if and only if for any $x \in C$ and $y \in F$
$$
\|f(x)-y\| \leq \|g(x) - y\| \hbox{ and } \| f(x_o)-y_o \| <   \| g(x_o)-y_o \|
$$
for some $ y_o \in F, $ $x_o \in C.$ It is  said that $ f \leq g$ if and only if $ f<g$ or $f=g.$ It is easy to see that $(N(F), \leq )$ is a partially ordered set. Observe that $ N(F) \neq \emptyset,$ since $ id_C \in N(F).$ 
Now we show that there exists a minimal element in  $(N(F), \leq ).$ First notice that for any $ f \in N(F)$ the set $ A_f = \{ g \in N(F): g \leq f\}$ is $ \tau$-closed in $N(F),$ which follows easily from definition of our partial ordering.
Hence by Lemma \ref{compact}, $A_f$ is a compact subset of $ N(F).$ Now let $G \subset N(F)$ be a chain. We show that there exists a smallest element in $G.$ Notice that if $ f_1, ...,f_n \in G$ then without loss of generality we can assume that
$ f_n \leq f_{n-1} \leq ...\leq f_1.$ Hence 
$$
\bigcap_{j=1}^n A_{f_j} = A_{f_n} \neq \emptyset. 
$$
Since the sets $ A_f $ are closed and $ N(F)$ is compact, $ \bigcap_{f \in G} A_f \neq \emptyset.$ Hence any $ g \in \bigcap_{f \in G} A_f$ is the smallest element in $G.$ By the Kuratowski-Zorn lemma, there exist $r \in N(F)$ a minimal element in 
$(N(F), \leq ). $ Let $ f \in N(F).$ Hence for any $ x \in C$ and $ y \in F,$
$$
\| y - (f \circ r)(x) \| \leq \|y - r(x) \| \leq \|y-x\|,
$$
which shows that $ f \circ r \leq r.$ If for some $ x_o \in C$ and $ y_o \in F,$ $\| y_o - (f \circ r)(x_o) \| < \|y_o - r(x_o) \|, $ then $ f \circ r \neq r$ and   $ f \circ r \leq r,$ a contradiction with the minimality of $r.$
\end{proof}
\begin{thm}
\label{existence1}
(compare with \cite{BR1}, Theorem 1)
Let $ F,C,X$ and $ N(F)$ be as in Lemma \ref{compact}. Assume that for any $x \in C$ there exists $ h \in N(F)$ such that $ h(x) \in F.$ Then there exists $ g \in N(F)$ such that for any $x \in C$ $g(x) \in F.$ In particular, $r(x) \in R_F(x)$ for any $ x \in C,$ where $r$ is a maximal element from Lemma \ref{maximal}. If $C =X, $ then $F$ is an existence set in $X.$ 
\end{thm}
\begin{proof}
Let $r \in N(F)$ be as in Lemma \ref{maximal}. We show that for any $x \in X,$ $ r(x) \in F.$ Notice that, since $ r(x) \in C,$  there exists $ h \in N(F)$ such that $ h(r(x)) \in F.$  Let $ y = h(r(x)).$ 
Observe that by Lemma \ref{maximal}, 
$$
0= \|h(r(x)) - y \| = \|y - r(x)\|.
$$
Hence $ r(x) = y \in F, $ as required. 
\end{proof}
\begin{thm}
\label{existence}
Let $ X$ be a reflexive space. For $ f \in X^* \setminus \{0\}$ define 
$$
G= \{ x \in X: f(x) \leq 0\}.
$$
If $ G$  is an existence set then $F=ker(f)$ is one-complemented in $X.$
\end{thm}
\begin{proof}
Since $ G$ is an existence set, for any $ x \in S(X),$ we can select  $P_ox \in R_G(x).$ Define a mapping $ P: X \rightarrow G$ by 
$P0 =0,$ $ Px = \|x\| P_o(x/\|x\|) $ for $ x \neq 0.$ Notice that $ Px \in R_G(x)\subset G$ for any $ x \in X.$ This is obvious if $x=0.$ If $x \neq 0,$ for any $ d \in G,$
$$
\| P_o(x/\|x\|) - d/\|x\|) \leq \|(x-d)/\|x\|\|
$$
and consequently 
$$
\| Px -d\| = \|\|x\| P_o(x/\|x\|) - d) \leq \|x-d\|. 
$$
In particular, for any $ d \in F=ker(f),$ $\|Px -d\| \leq \|x-d\|$ and $ P|_{F} = id_{F}.$  This shows that $ P \in N(F)$ (in our case $C=X).$ Now we show that for any $ x \in X$ there exists $Q \in N(F)$ 
such that $ Qx \in F.$ Fix $ x \in X.$  If $f(x)=0, $ then $ Id_F (x) \in F.$ 
Now assume that $f(x) >0.$ If $ f(Px)=0,$ then 
$Px \in F.$ If $f(Px)<0,$  then there exists $ 0 < \alpha < 1$ such that $f(\alpha x + (1-\alpha) Px)=0.$ Since $ N(F)$ is convex, $ Q = \alpha Id_X + (1-\alpha)P \in N(F)$ and $Qx \in F.$ Now assume that $ f(x) <0$ 
and set $ G_1 = \{ x \in X: f(x) \geq 0 \}. $ By Lemma \ref{basicp} for any $x \in X,$  $ P_1x = -P(-x) \in R_{G_1}(x) \subset G_1.$ Since $ F \subset G_1,$ for any $d \in F,$ $ \|P_1x -d \| \leq \|x-d\|$ and $P_1|_F = id_F.$  This shows that $ P_1 \in N(F).$
If $ f(P_1x) =0, $ then $ P_1x \in F.$ If $ f(P_1 x) >0,$   then there exists $ 0 < \alpha < 1$ such that $f(\alpha x + (1-\alpha) P_1x)=0.$ Since $ N(F)$ is convex, $ Q_1 = \alpha Id_X + (1-\alpha)P_1 \in N(F)$ and $Q_1x \in F.$ By Theorem \ref{existence1}
applied to $C=X,$ there exists $Q:X \rightarrow ker(f)$ such that $ Q|_{ker(f)} = id|_{ker(f)} $ and and $ \|Qx-d\| \leq \|x-d\|$ for any $ x \in X$ and $d\in ker(f).$ Hence $ker(f)$ is an existence set. By Lemma \ref{basicp}, 
$ker(f)$ is one-complemented in $X.$   
\end{proof}
\begin{thm}
\label{contractive}
Let $X$ be a Banach space and let $ f \in X^* \setminus \{0\}.$ If $ ker(f)$ is one-complemented in $X,$ then for any $ d \in \mathbb{R}$ the set $ G_d = \{x \in X: f(x) \leq d\}$ is a contractive subset of $X.$    
\end{thm}
\begin{proof}
By Lemma \ref{basicp}, we can assume that $d=0.$  Fix $P \in {\cal P}(X, ker(f)),$ $ \|P\|=1.$ Define $ Q : X \rightarrow G_o$ by $ Qx = x$ if $ f(x) \leq 0$ and $ Qx = Px$ if $ f(x) >0.$   We show that for any $ x,z \in X$ $ \|Qx - Qz\| \leq \|x-z\|.$
If $f(x) \leq 0$ and $ f(z) \leq 0,$ then
$$ 
\|Qx -Qz\| = \| x-z \|.
$$ 
If $ f(x) \geq 0 $ and $ f(z) \geq 0,$ then 
$$ 
\|Qx -Qz\| = \| Px-Pz \| \leq \| x-z \|. 
$$ 
Now assume that $ f(x) >0$ and $ f(z) < 0.$
Let  $ y \in X$ be so chosen that $ Py=0$ and $ f(y) =1.$ Then, it is easy to see that for any $w \in X,$ $ Pw = w - f(w) y.$ Hence if $ w = \alpha y + v$ where $ v \in ker(f)$  and $ \alpha \in \mathbb{R},$
then 
$$
\|Pw\|= \|\alpha Py + Pv \| = \|v \| \leq \|w\| = \|\alpha y + v \|.  
$$
This means that for any $ v \in ker(f)$ and $ t \geq 0$ the function $g_v(t) = \|ty+v\| $ satisfies 
$$ 
g_v(t) \geq g_v(0) =\|v\|.
$$
Sinec $g_v$ is a convex function, $ g_v$ is increasing in $ [0, +\infty). $ Since $f(x) >0,$ if $x =\alpha y +v,$ then $ f(x) = \alpha >0.$
Analogoulsy, since $f(z) <0,$ if $z =\beta y +w,$ then $ f(z) = \beta <0.$ Notice that, since the function $ g_{v-w}$ is increasing in $[0, + \infty ),$  
$$
\|Qx - Qz\| = \| Px -z \|= \| v - (\beta y + w\| = \| -\beta y + (v-w)\|
$$
$$
= g_{v-w}(-\beta) \leq g_{v-w}(-\beta + \alpha ) =
$$
$$
= \|(-\beta + \alpha )y + v-w\| = \|\alpha y + v - (\beta y + w) \| = \|x-z\|,
$$
as required. 
Since $ Q|_{G_o} = id_{G_o},$ $G_o$ is a contractive set. The proof is complete. 
\end{proof}
\begin{lem}
 \label{mink}
Let $X$ be a Banach space and let $ C \subset X$ be a closed, convex and bounded set such that $ 0 \in int(C).$ Define for $ x \in X, $ 
 $$
 f(x) = \inf \{ t>0: x/t \in C\}.
 $$
 Then $f$ is a convex and continuous function.
\end{lem}
\begin{proof}
First we show that for any $ x_1, x_2 \in X,$ $ f(x_1+x_2) \leq f(x_1) + f(x_2).$ Fix $ x_1, x_2 \in X$ and $ \epsilon > 0.$ Then there exists $ t_1 >0$ and $ t_2 >0$ satisfying 
 $t_1 < f(x_1) + \epsilon $ and $t_2 < f(x_2) + \epsilon  $ such that $ \frac{x_i}{t_i} \in C $ for $ i=1,2.$ Notice that 
$$
\frac{x_1 +x_2}{t_1+t_2} = \frac{t_1}{t_1+t_2} \frac{x_1}{t_1} + \frac{t_2}{t_1+t_2} \frac{x_2}{t_2}.
$$
Since $ C$ is convex, by definiton of $ t_1$ and $t_2,$ $\frac{x_1 +x_2}{t_1+t_2} \in C.$ This shows that 
$$ 
f(x_1+x_2) \leq f(x_1) + f(x_2) + 2\epsilon
$$
and consequently  $f$ is a subadditive function. Moreover, it is easy to see that 
$ f(ax) = af(x)$ for any $ a\geq 0$ and $ x \in X. $ Hence for any $ a,b \geq 0,$ $ a+b=1,$ and $x,y \in X$   
$$
f(ax + by) \leq f(ax) + f(by) \leq af(x) + bf(y),
$$
which shows that $f $ is a convex function. Now we prove that $f$ is continuous. Let $ \|x_n -x\| \rightarrow 0.$ Fix $\epsilon >0$ and $ t>0$ such that $ f(x) < t < f(x) + \epsilon.$ Since $ 0 \in int(C),$ there exists $ r>0$ such that 
$B_o(x/t,r) \subset C.$ Hence for $ n \geq n_o,$ $ x_n/t \in B_o(x/t,r) \subset C.$ Hence for any $ n \geq n_o,$ $ f(x_n) \leq t \leq f(x) + \epsilon$ and conseqently, 
$$
\limsup_{n} f(x_n) \leq f(x).
$$
Now, assume on the contrary that
$$
0 \leq d= \liminf_{n} f(x_n) < f(x).
$$
Fix $ d < t < f(x).$ Hence $ f(x) >0, $ so $ x \neq 0.$ Then there exists a subsequence $ x_{n_k} $ such that $ f( x_{n_k}) <t$ and consequently  $ x_{n_k}/t \in C.$ Since $ C$ is closed, $ x/t \in C. $ Hence $ f(x) \leq t,$ which is a contradiction. 
Finally we get
$$
\limsup_{n} f(x_n) \leq f(x) \leq \liminf_{n} f(x_n)
$$
which shows that $ f(x) = \lim_{n} f(x_n),$ as required. 
\end{proof}
\begin{rem}
If $C = -C$ then the function $f$ from Lemma \ref{mink} is the Minkowski functional of $C.$
\end{rem}

\begin{lem}
\label{inters}
Let $X$ be a Banach space and let $C \subset X$ be a closed, convex and bounded set such that $ 0 \in int(C).$ Let $ \{ x_n\} \subset \sigma(C)$ be a dense subset of $ \sigma(C)$ such that for any $ n \in \mathbb{N}$ there exists exactly one 
supporting functional $f_n$ for $x_n,$ with $ \|f_n\|=1.$ 
Assume that $ f_n(x_n) = d_n $ and $ f_n(v) \leq d_n$ for any $ v \in V.$ Put 
$$
D_n = \{ x \in X: f_n(x) \leq d_n\}.
$$
Then $ C = \bigcap_{n \in \mathbb{N}} D_n.$
\end{lem}
\begin{proof}
By definition of $ f_n$ and $D_n ,$  $ C \subset  \bigcap_{n \in \mathbb{N}} D_n.$ Now assume on the contrary that  there exists $y_o \in  \bigcap_{n \in \mathbb{N}} D_n \setminus C.$ Since $ 0 \in int(C),$ there exists $t \in (0,1)$ such that 
$ ty_o \in \sigma(C).$ Fix $ \{ y_n \} \subset  \{x_n\} $ such that 
\begin{equation} 
\label{dense}
\| y_n - ty_o\| \rightarrow 0.
\end{equation} 
This is possible, since  $ \{ x_n\} $ is a dense subset of $ \sigma(C).$ Let $d_n$ and $f_n$ be so chosen that  
$ f_n(y_n) = d_n$ and $ V \subset D_n = \{ x \in X: f_n(x) \leq d_n\}.$ Now we show that $\{d_n\}$ is bounded. Assume on the contrary that this is not true. Without loss of generality, passing to a subsequence, if necessary,  we can assume that 
$ d_n \rightarrow +\infty .$ Let $R >0 $ be so chosen that $ C \subset B(0,R).$  Since $ f_n(y_n) =d_n, $ and $ y_n \in C, $ for $ n \geq n_o,$
$$
\sup\{ f_n(z): z \in B(0,R) \} \geq f_n(y_n) \geq 2R.
$$
Hence  
$$
\sup\{ f_n(z): z \in B(0,1) \} \geq 2
$$
and consequently, $ \| f_n\| >2 $ for $ n \geq n_o$ which is a contradiction. Now we show that $ e_o = \liminf_n d_n > 0.$ Assume on the contrary that  $ \liminf_n d_n = 0.$ Without loss of generality, passing to a subsequence, if necessary,  we can assume that 
$ d_n \rightarrow 0 .$ Fix $ r>0$ such that $B(0,r) \subset int(C).$ Fix $ n_o \geq 2 $ such that $ \frac{r}{n_o} <1. $ Fix $ k_o$ such that $ 0 \leq d_k < \frac{r}{n_o}$ for $k \geq k_o.$ Fix $ x \in B(0,r).$ If $ f_k(x) \geq 0,$ then
$$
0 \leq f_k(x) \leq d_k < \frac{r}{n_o} .
$$
If $ f_k(x) < 0, $ then $f_k(-x) \geq 0 $ and $ f_k(-x) \leq d_k <  \frac{r}{n_o}.$ Hence 
$$
\sup\{ |f_k(x)| : x \in B(0,r)\} \leq \frac{r}{n_o}
$$
and consequenntly $ \| f_k \| \leq \frac{1}{n_o} < 1$ which is a contradiction. Without loss of generality, passing to a subsequence if necessary, by (\ref{dense}), we can assume that $ \|y_n - ty_o \| \rightarrow 0$ and $ d_n \rightarrow e_o > 0.$  
Notice that 
$$
f_n(ty_o)= f_n(y_n) - f_n(y_n - ty_o)= d_n - f_n(y_n - ty_o) \geq d_n - \|y_n - ty_o\|.
$$
Hence 
$$
\liminf_n f_n(y_o) = \frac{\liminf_n f_n(ty_o)}{t} \geq \frac{\lim_n d_n}{t} = \frac{e_o}{t}.
$$
Since $ t \in (0,1)$ and $ e_o >0,$ $f_n(y_o) >d_n$ for $n \geq n_o.$ Hence $y_o \notin  \bigcap_{n \in \mathbb{N}} D_n,$ which is a contradiction. The proof is complete.
\end{proof}
\begin{lem}
 \label{dir1}
 Let $X$ be a reflexive Banach space. Let $ \{F_{t}\}_{t \in T}$ be a family of convex existence sets directed by $ \subset .$ Then $ F = cl(\bigcup_{t \in T} F_t)$ is also an existence set.  
\end{lem}
\begin{proof}
Since $ \{F_{t}\}_{t \in T}$ is directed by $ \subset ,$ for any $ t_1, t_2 \in T$ there exists $ t_3 \in T$ such that $ F_{t_1} \cup F_{t_2} \subset F_{t_3}.$ Define a partial ordering  $ \leq $ in  $ T$ by 
$t_1 \leq t_2 $ provided $ F_{t_1} \subset F_{t_2}.$  Fix $ F_{t_o} = F_o \in  \{F_{t}\}_{t \in T}.$ Since  $\{F_{t}\}_{t \in T}$ is directed by $ \subset ,$ $ F = cl(\bigcup_{t \geq t_o} F_t).$ By the axiom of choice, for any $ t \geq t_o$ we can select 
$ P_t \in N(F_t) $ (see Lemma \ref{compact}) such that $ P_t (x) \in R_{F_t}(x)$ for any $ x \in X.$  Observe that $ \{ P_t \}_{t \geq t_o} \subset N(F_o).$ Set for any $ t \geq t_o $ $ A_t = cl(\{P_s : s \geq t\})$ where the closure is taken with respect to 
the topology in $ N(F_o)$ 
defined in Lemma \ref{compact}. Since  $\{F_{t}\}_{t \in T}$ is directed by $ \subset ,$ for any $ n \in \mathbb{N}$ and $t_1,...t_n \geq t_o$ there exists $ u \in T$ such that 
$$
\bigcap_{i=1}^n A_{t_i} \supset A_{u} \neq \emptyset.
$$
By the proof of Lemma \ref{compact}, $ \bigcap_{t \geq t_o} A_t \neq \emptyset.$ Fix $ P \in  \bigcap_{t \geq t_o} A_t.$ We show that for any $ x \in X,$ $Px \in R_F(x).$ Fix $x \in X $ and $ d \in \bigcup_{t \geq t_o} F_t.$ Then there exists $ s \geq t_o$ such that 
$ d \in F_s$ and consequently $d \in F_u $ for $ u \geq s.$ Since $ P \in A_s,$
$$
\|Px  - d \| \leq \liminf_{u \geq s} \|P_u x - d\| \leq \|x-d\|.
$$
If $ d \in F,$ then there exists a sequence $ \{d_n\} \subset \bigcup_{t \geq t_o} F_t$ such that $ \|d_n -d\| \rightarrow 0.$ By the previous reasoning, 
$$
\|Px -d\| = \lim_n \|Px -d_n \| \leq \lim_n \|x -d_n \| = \| x-d\|,
$$
as required.
\end{proof}
\begin{cor}
\label{cone}
 Let $ X$ be a reflexive Banach space and let $ F$ be a convex existence set. Define for any $ v \in F$ and $ t>0 $ $ F_t =  (1-t)v + tF $ and let $ C_v =  cl(\bigcup_{t >0} F_t).$ Then $F$ is an existence set. 
\end{cor}
\begin{proof}
By Lemma \ref{basicp}, $F_t$ is an existence set for any $ t>0.$ Since $F_t \subset F_s$ for $ t \leq s,$ by Lemma \ref{dir1} $ C_v$ is an existence set too.   
\end{proof}
\begin{lem}
 \label{dir2}
 Let $X$ be a reflexive Banach space. Let $ \{F_{t}\}_{t \in T}$ be a family of convex existence sets directed by $ \supset .$ If $ F = \bigcap_{t \in T} F_t \neq \emptyset,$ then $F$ is also an existence set.  
\end{lem}
\begin{proof}
Since $ \{F_{t}\}_{t \in T}$ is directed by $ \supset ,$ for any $ t_1, t_2 \in T$ there exists $ t_3 \in T$ such that $ F_{t_1} \cap F_{t_2} \supset F_{t_3}.$ Define a partial ordering $ \leq $ in  $ T$ by 
$t_1 \leq t_2 $ provided $ F_{t_1} \supset F_{t_2}.$  By the exiom of choice, for any $ t \in T$ we can select 
$ P_t \in N(F_t) $ such that $ P_t (x) \in R_{F_t}(x)$ for any $ x \in X.$  Observe that $ \{ P_t \}_{t \in T} \subset N(F).$ Set for any $ t \geq t_o $ $ A_t = cl(\{P_s : s \geq t\})$ where the closure is taken with respect to the topology in $ N(F)$ 
defined in Lemma \ref{compact}. Since  $\{F_{t}\}_{t \in T}$ is directed by $ \supset ,$ for any $ n \in \mathbb{N}$ and $t_1,...t_n \geq t_o$ there exists $ u \in T$ such that 
$$
\bigcap_{i=1}^n A_{t_i} \supset A_{u} \neq \emptyset.
$$
By the proof of Lemma \ref{compact}, $ \bigcap_{t \in T} A_t \neq \emptyset.$ Fix $ P \in  \bigcap_{t \in T} A_t.$ We show that for any $ x \in X,$ $Px \in R_F(x).$ Fix $x \in X $ and $ d \in F.$ Hence 
$ d \in F_s$ for any $s \in T$ and consequently
$$
\|Px  - d \| \leq \liminf_{s} \|P_s x - d\| \leq \|x-d\|.
$$
To finish the proof we need to show that $ Px \in F$ for any $ x \in X.$ Assume on the contrary that there exists $x_o \in X$ such that $Px_o \notin F.$ Then there exists $ t_o \in T$ such that $ Px_o \notin F_{t_o}.$ Hence $ Px_o \notin F_t $ for any $ t \geq t_o.$
Since $ P_t x_o \in F_t \subset F_{t_o}$ for $ t \geq t_o,$ by the Mazur Theorem $Px_o \in F_{t_o},$ which is a contradiction. 
\end{proof}
\section*{Section 2}
Now we state the main result of this paper.
\begin{thm}
\label{main}
Let $X$ be a reflexive, separable Banach space. Let $ F \subset X$ be a nonempty, closed, bounded and convex existence set with nonempty interior in $X.$. Then $F$ is an intersection of a countable family of contractive half-spaces. 
\end{thm}
\begin{proof}
 By Lemma \ref{basicp} we can assume that $ 0 \in int(F).$ Define a function $f:X \rightarrow \mathbb{R}$ by 
 $$
 f(x) = \inf \{ t>0: x/t \in C\}.
 $$
 By Lemma \ref{mink} $f$ is a convex and continuous function. By the Mazur Theorem $f$ is Gateaux differentiable on a countable, dense subset of $X.$ Since $f(tx) = t f(x)$ for any $x \in X$ and $ t>0,$ there exists a dense countable subset  $Z= \{z_n\}$ of 
$ \sigma(C)$  such that $f$ is is Gateaux differentiable at any $z \in Z.$ Let for $n \in \mathbb{N}$ $ D_n = cl(\bigcup_{t>0}F_{t,n}), $ where $ F_{t,n} = \{ (1-t)z_n + tF\}.$ By Corollary \ref{cone}, $ D_n $ is an existence set for any $n \in \mathbb{N}.$
 Since  $f$ is is Gateaux differentiable at $z_n,$ there exists $ f_n \in S(X^*) $ and $ d_n \in \mathbb{R} $ such that $ D_n = \{ x \in X: f_n(x) \leq d_n\}$ and $ f_n(x_n)=d_n.$ (Since $ 0 \in int(C),$ $ d_n >0.)$ By Theorem \ref{existence} and Lemma \ref{basicp},
 $ker(f_n)$ is one-complemented in $X.$ By Theorem \ref{contractive}, 
 $ D_n$ is a contractive half-space. By Lemma \ref{inters}, $ F = \bigcap_{n=1}^{\infty}D_n,$ as required. 
\end{proof}
Now we will show a sufficient condition under which intersections of  countable families of contractive sets are contractive.
\begin{thm}
\label{countable}
Let $(X, \|\cdot \|) $ be a reflexive Banach space. Let $ \| \cdot \|_n$ be a sequence of strictly convex norms on $X$ such that there exists a sequence $\{s_n\}$ of nonegative numbers, $ s_n \rightarrow 0$ satisfying for any $ n\in \mathbb{N}$ 
and $x \in X$
\begin{equation}
\label{equiv}
(1-s_n) \|x \|_n \leq \|x\| \leq (1+s_n) \|x\|_n
\end{equation}
Assume that $\{F_k\} $ is a countable family of convex sets such that for any $k,n \in \mathbb{N}$ $ F_k$ is a contractive set with respect to $\| \cdot \|_n. $ Assume that $ F = \bigcap_{k \in \mathbb{N}} F_k \neq \emptyset.$ 
Then $F$ is a contractive subset of $X$ with respect to $ \| \cdot \|.$
\end{thm}
\begin{proof} 
By Lemma \ref{basicp}, we can asume that $ 0 \in F.$
Fix $ n \in \mathbb{N}.$ Let for $ k \in \mathbb{N}$ $P_{k,n}: X \rightarrow F_k$ be a contractive mapping with respect to $ \| \cdot \|_n.$ Fix a sequence of positive numbers $ \{c_k\}$ such that $ \sum_{k=1}^{\infty} c_k =1.$ 
Define $P_n : X \rightarrow X$ by 
$$
P_nx = \sum_{k=1}^{\infty} c_k P_{k,n} x.
$$
(Since  $ \sum_{k=1}^{\infty} c_k =1,$ $P_n$ is well-defined.) Now we show that $ Fix(P_n) =F.$ By definition of $ P_{k,n} , $ $ F \subset Fix(P_n).$ Now asume that $ x \in Fix(P_n) \setminus \{0\}.$ Since $ 0 \in F, $ 
for any $ k \in \mathbb{N} $ $\| P_{k,n}(x)\|_n \leq \|x \|_n.$ Fix $f_n \in S(X^*)$ (with respect to $ \| \cdot \|_n)$ such that $ f_n(x) = \|x\|_n.$ Notice that 
$$
\|x \|_n = f(x) = f_n(\sum_{k=1}^{\infty} c_k P_{k,n} x) = \sum_{k=1}^{\infty} c_k f_n(P_{k,n} x) \leq \sum_{k=1}^{\infty} c_k \|P_{k,n} x \|_n \leq \|x\|_n.
$$
Hence for any $ k \in \mathbb{N}$ $ \|P_{k,n}x \|_n = \|x\|_n$ and $ f_n(P_{k,n}x) = \|x\|_n.$ Hence for any $ k \in \mathbb{N}$ $ \frac{P_{k,n}x}{\|P_{k,n}x \|_n}$ is a norming point for $f_n.$ Since $(X \| \cdot \|_n)$ is strictly convex, 
$f_n$ has exactly one norming point with respect to $ \| \cdot \|_n.$ Hence for any $ k \in \mathbb{N},$ $ P_{k,n}x = x,$ which shows that $ x \in F.$ By (\cite{BR1}, Theorem 2 and Lemma 3),  $F = Fix(P_n)$ is a contractive subset of  $(X, \|\cdot \|_n).$ 
Let $ Q_n:X \rightarrow F$ be a contractive mapping with respect to $ \| \cdot \|_n.$ Define for any $ M \geq 1 $
$$
N_M(F) = \{ f:X \rightarrow X: \| y - f(x) \| \leq M \|y-x\| \hbox{ for any } y\in F, x \in X \}.
$$
Reasoning as in Lemma \ref{compact}, we can show that $N_M(F)$ is a compact set with respect to the topology of weak pointwise convegence. 
Put 
$$ 
M = sup \{ \frac{1+s_n}{1-s_n} : n \in \mathbb{N}\}.
$$ 
Notice that for any $x \in X$ and $ y \in F,$  
$$
\|Q_n x - y \| \leq (1+s_n) \|Q_nx-y \|_n \leq (1+s_n) \| x-y \|_n \leq \frac{1+s_n}{1-s_n} \|x-y\| \leq M \|x-y\|.
$$
Hence $Q_n \in N_M(F)$ for any $ n \in \mathbb{N}.$ Renasoning as in Lemma \ref{maximal}, we can show that $ \bigcap_{n\in N} A_n \neq \emptyset , $ where $ A_n = cl(\{ Q_k: k \geq n \}).$ 
Fix $ Q \in \bigcap_{n\in N} A_n.$ We show that $Q$ is a contractive mapping from $ X $ onto $F$ with respect to $ \| \cdot \|.$ 
Notice that for any $ x , y \in X$
$$
\|Q_n x - Q_n y\|_n \leq \|x-y\|_n \leq \frac{\|x-y\|}{1-s_n}. 
$$
Since $ s_n \rightarrow 0,$ 
$$
\liminf_n \frac{\|Q_n x - Q_ny\|_n}{1-s_n}  \leq \|x-y\|.
$$
Since $ Q \in \bigcap_{n\in N} A_n,$ this shows that 
$$
\|Qx - Qy\| \leq \liminf \|Q_n x - Q_n y\| \leq \liminf_n \frac{\|Q_n x - Q_n y\|_n}{1-s_n} \leq \|x-y\|,
$$
as required.
Moreover, since for any $ n \in \mathbb{N}$ $Q_nx \in F$ and $F$ is convex, by the Mazur Theorem, $ Qx \in F.$ Since for any $n \in \mathbb{N}$ and $x \in F, $ $ Q_nx =x, $ $Qx=x.$ The proof is complete.
\end{proof}
\begin{rem}
\label{r1}
Observe that, in general, the countable intersection of contractive sets need not to be even an existence set. In (\cite{KL}, Example 2.10), it was shown that $ F = ker(f_1) \cap ker(f_2) \subset l_{\infty}^{(4)}$ is not an existence set.
Here $ f_1 = (1,0,0,0) $ and $f_2 = (1/2, 1/6,1/6,1/6).$ (We understand that $ f_i(x) = \sum_{j=1}^4 x_jf_j,$ for $ x \in   l_{\infty}^{(4)}$ and $i=1,2.)$ Observe that $ F= \bigcap_{n=1}^{\infty} F_n,$ where
$$
F_n=\{ x \in l_{\infty}^{(4)}: |f_i(x)| \leq 1/n, i=1,2\}. 
$$ 
By Lemma \ref{dir2} and Theorem \ref{basicp} for any $ n \in \mathbb{N} $ $F_n$ is not an existence set. Observe that for any $ n \in \mathbb{N}$
$F_n= \bigcup_{k=1}^{\infty} F_{n,k},$ where 
$$
F_{n,k} = \{ x \in F_n : |x_i| \leq k \hbox{ for } k=1,...,4\}.
$$
By Lemma \ref{dir1} and Theorem \ref{basicp}, for any $ k,n \in \mathbb{N},$ $F_{n,k}$ is not an existence set too.
\end{rem}
An immediate consequence of Theorem \ref{main} and Theorem \ref{countable} is
\begin{thm}
\label{s}
Let $(X, \| \cdot \|)$ and $F$ be as in Theorem \ref{main}. Assume that $ (X, \|\cdot \|)$ satisfies the assumptions of Theorem \ref{countable}. Then the following conditions are equivalent:

a. $ F$ is an existence set;

b. $F$ is a contractive set; 

c. $F$ is an intersection of a countable number of contractive half-spaces.
\end{thm}
\begin{proof}
Assume that $F$ is an existence set. By Theorem \ref{main},  $F$ is an intersection of a countable number of contractive half-spaces. By Theorem \ref{countable}, $F$ is a contractive set. 
Since each existence set is contractive, the proof is complete.
\end{proof}
Now we present two applications of Theorem \ref{s}.
\begin{thm}
\label{l1}
Let $ X= l_1^{(n)}$ and let $F \subset X$ be a convex bounded set with nonempty interior in $X.$ Then the following conditions are equivalent

a. $ F$ is an existence set;

b. $F$ is a contractive set; 

c. $F$ is an intersection of a countable number of contractive half-spaces.
\end{thm}
\begin{proof}
 Assume that $F$ is a convex existence set in $X.$ By Theorem \ref{main}, $F$ is an intersection of a countable family $ \{F_k\} $ of half-spaces. By Theorem \ref{existence}, each $F_k$ is determined by $ f_k \in S(X^*)$ such that $ker(f_k)$ 
 is one-complemented in $X.$ By\cite{BC}, each $f_k$ has at most two coordinates different from zero. By \cite{BO} and \cite{BP},  for each $k,$ $ker(f_k)$ is also one-complemented in $ l_p^{(n)}$ for $ p\geq 1. $ Fix a sequence $ p_l > 1$ $ p_l \rightarrow 1.$
 By Theorem \ref{contractive}, for any $k,l$  $ F_k$ is a contractive subset of $l_{p_l}^{(n)}.$  Applying Theorem \ref{s} to $\| \cdot \|_l = \| \cdot \|_{p_l}$ and $ \{F_k\}$ we get the result.  
\end{proof}
\begin{thm}
\label{linf}
Let $ X= l_{\infty}^{(n)}$ and let $F \subset X$ be a convex set.  Assume that $F = \bigcap_{k=1}^{\infty} F_k,$ where for each $k$ $F_k$ is a half space determined by $f_k \in S(X^*) $ having at most two coordinates 
different from zero. Then the following conditions are equivalent

a. $ F$ is an existence set;

b. $F$ is a contractive set; 

\end{thm}
\begin{proof}
Fix a sequnece $p_l  \rightarrow +\infty. $ By \cite{BO} and \cite{BP},  and Theorem \ref{contractive}, for any $k,l$  $ F_k$ is a contractive subset of $l_{p_l}^{(n)}.$ Reasoning as in Theorem \ref{l1} we get the result. 
\end{proof}
Now we present a sufficient condition (in which strict convexity is not assumed), under which the intersection of a countable family of contractive half-spaces is a contractive set.
\begin{thm}
 \label{suff}
Let $X$ be a reflexive Banach space satisfying (CFPP) (see Def. \ref{fixed}). Assume that $F = \bigcap_{k=1}^{\infty} F_k,$ where for each $k$ $F_k$ is a contractive half-space determined by $f_k \in S(X^*).$ 
(By Theorem \ref{existence}, for each $k$ $ ker(f_k) $ is one complemented in $X.)$
Let $ P_k $ be a norm one projection from $ X $ onto $ker(f_k).$ (By \cite{BC}, for any $x \in X$ $ P_kx = x - f_k(x) y_k$ where $ y_k \in X $ satisfies $ f_k(y_k)=1.)$ Assume furthermore that there exists a sequence $ \{ d_k\}$ of nonegative numbers such that 
$ F_k = \{ x \in X: f_k(x) \leq d_k\}.$ 
If $ 0 \notin cl(conv(\{y_k\}_{k \in \mathbb{N}})), $ then $F$ is a contractive set.   
\end{thm}
\begin{proof}
We follow the idea included in \cite{BR1}, Lemma 3).
Fix a sequence $ a_n$ of positive numbers such that $ \sum_{n=1}^{\infty} a_n =1.$ By the proof of Theorem \ref{contractive} and Lemma \ref{basicp} the mapping $ Q_n: X \rightarrow F_n$ defined by $ Q_nx = x - f_n(x-z_n)y_n$ if $ f_n(x) > d_n$  and 
$Q_nx =x$ in the opposite case, is a contractive projection from $X$ onto $F_n.$ (Here for each $ n \in \mathbb{N}$ $z_n$ is so chosen that $ f_n(z_n)=d_n.)$  Define $ Q: X \rightarrow X$ by 
$$
Qx = \sum_{n=1}^{\infty} a_n Q_n .
$$
It is easy to see that $Q$ is a nonexpansive mapping. 
We show that $ Fix(Q_n) = F.$ It is easy to see that $ F \subset Fix(Q).$ Assume on the contrary that there exists $ x \in Fix(Q) \setminus F.$ Since $F = \bigcap_{k=1}^{\infty} F_k,$ the set 
$$ 
Z= \{ n \in \mathbb{N}: Q_k(x) \neq x \} \neq \emptyset.
$$
By defintion of $Q$
$$
x = \sum_{k \notin Z}a_k x + \sum_{k \in Z} a_k Q_kx.
$$
(We assume that $\sum_{k \notin Z} = 0$ if $ Z = \mathbb{N}.)$ 
By defintion of $ Q_k,$
$$
0 = \sum_{k \in Z} a_k f_k(x-z_k)y_k.
$$
Since $ a_k >0$ and $ f_k(x-z_k) >0$ for $k \in Z, $ $ 0 \in cl(conv(\{y_k\}_{k \in \mathbb{N}})),$ which is a contradiction.
\end{proof}
\begin{rem}
\label{CFPP}
In \cite{BR1}, Theorem 4, sufficient conditions for a Banach space $X$ satisfying (CFPP) are presented. In particular, any finite-dimensional Banach space satisfies (CFPP).
\end{rem}
\begin{ex}
\label{onecoor}
Let $ X = l_{\infty}^{(n)}.$ Let for $k \in \mathbb{N},$ $ f^k \in S(X^*),$ $ f^k =  (f^k_1,...,f^k_n)$ be so chosen that 
\begin{equation}
\label{dom}
|f^k_j| \geq \sum_{i\neq j} |f^k_i|
\end{equation}
for $ j \in \{1,...,\}$ depending on $k.$   
By \cite{BC}, for any $k,$ $ ker (f^k)$ is one-complemented in $X.$ Fix a sequence of nonegative numbers $\{d_k\}.$ By Theorem \ref{contractive} the half spaces $ F_k = \{ x \in X: f^k(x) \leq d_k\} $ are contractive subsets of $X.$
Assume that $F = \bigcap_{k=1}^{\infty} F_k,$ is a nonempty set . Let for any $ j \in \{1,...,n\},$ $ Z_j = \{ k \in \mathbb{N}:  |f^k_j| \geq \sum_{i\neq j} |f^k_i|\}.$ Assume that for some $ j \in \{ 1,..., n\}$ such that $ Z_j \neq \emptyset, $ 
$ f^k_j > 0$ for any $ k \in Z_j.$ By \cite{BC} and the proof of Theorem \ref{suff}, for each $k \in \mathbb{N},$ $ Q_k x = x - f^k(x-z_k) y_k, $ where $ y_k =(0, (\frac{1}{f^k_j})_j,..., 0)$ where $ j \in \{1,...,n\}$ satisfies (\ref{dom}).
Hence by our assumption on $ Z_j, $  $ 0 \notin cl(conv(\{y_k\}_{k \in \mathbb{N}})).$ By Theorem \ref{suff}, $F$ is a contractive subset of $ X.$ Observe that if for some $ k $ $f^k$ has more than two coordinates different from zero, 
this result cannot be deduced from Theorem \ref{linf}.
\end{ex}
\begin{rem}
 \label{r2}
The assumption that $F$ is a bounded set in Theorem \ref{s} can be weakend. If we assume that $ F = cl(\bigcup_{t \in T} F_t)$ where $\{F_t\}_{t \in T} $ is a directed by $ \subset $ family of bounded and convex existence sets such that
$int(F_{t_o}) \neq \emptyset $ for some $ t_o \in T,$ by Lemma \ref{dir1} the conditions  (a) and (b) from Theorem \ref{s} are equivalent. If $ T = \mathbb{N}$ and $int(F) \neq \emptyset,$  then by 
the Baire Property $int(F_{t_o}) \neq \emptyset $ for some $ t_o \in \mathbb{N}.$
\end{rem}
In particular, we can prove
\begin{thm} 
\label{s1}
Let $(X, \|\cdot\|)$ be a strictly convex, reflexive and separable Banach space and let $F \subset X$ be a bounded set with nonempty interior (compare with Remark \ref{r2}.) Then the following conditions are equivalent  

a. $ F$ is an existence set;

b. $F$ is a contractive set; 

c. $F$ is an intersection of a countable number of contractive half-spaces.

d. $F$ is an optimal set. 
\end{thm}
\begin{proof}
If $ F$ is an existence set, then, by definition, $F$ is an optimal set. Since $ X$ is strictly convex, if $F$ is an optimal set, by \cite{BM}, $F$ is a convex set. By Theorem \ref{optimal}, 
$F$ is an existence set. Hence by Theorem \ref{s}, the proof is complete. 
\end{proof}
\begin{rem}
\label{r3}
There exists Banach spaces such that there is no bounded, convex sets with nonempty interior being existence sets. For example, if we take $ X = L_p[0,1]$ for $ 1 < p < +\infty,$ then there is no $ F \subset X,$ $ F$ convex with $ int(F) \neq \emptyset$
being an existence set. Indeed, if such an $F$ exists, then by Theorem \ref{main} $F$ would be an intersection of a countable family of contractive half-spaces $\{D_n\}.$ By Theorem \ref{existence}, each $D_n$ is determined by $f_n \in S(X^*)$ with $ker(f_n)$
being one-complemented. But by \cite{BM}, (see also \cite{FR}, \cite{LS}, \cite{R} and \cite{RO}),each hyperplane in $X$ is not one complemented.
\end{rem}
\section*{Section 3}
Now, we prove some results on existence sets $F$ without assumptions that they have nonempty interior in the whole space $X.$ To the end of this section, if is not otherwise stated, we assume that $int_Z(F) \neq \emptyset $ where $ Z = cl(span(F))$ and $ int_Z(F)$ 
means the interior with respect to $Z.$
\begin{rem}
\label{r4}
It may happen that $int_Z (F) = \emptyset.$ For example, if we take $ X = L_p[0,1]$ for $ 1 \leq p < +\infty,$ and $ F =\{ f \in X: f \geq 0\} $ then it is easy to see that $ Span(F) = X$ and $ int_X(F) = \emptyset.$
Observe that $F$ is a contractive subset of $X.$ Indeed, it is easy to see that the mapping $ Pf = f\chi(D_{f,+}) $ is a contractive projection onto $F.$ Here  $\chi(D_{f,+})$ denote the characteristic function of the set 
$ D_{f,+} = \{ t \in [0,1]: f(t) \geq 0\}.$ If $X$ is finite-dimensional, then $ int_Z(F) \neq \emptyset$ for any $F \subset X,$ $ F \neq \emptyset.$ 
\end{rem}
We start with 
\begin{thm}
\label{int1}
Let $ X$ be a Banach space and  let $F \subset X$ be a convex and bounded (compare with Remark \ref{r2}) existence set. Assume that $ Z = cl(span(F))$ satisfies the assumption of Theorem \ref{s}. If $Z$ is a contractive subset of $X$ then
$F$ is a contractive subset of $X.$
\end{thm}
\begin{proof}
Since $F$ is an existence set in $X,$ $F$ is an existence set in $Z.$ By Theorem \ref{s} applied to $Z$ there exists a contractive projection $Q:Z \rightarrow F.$ Let $P: X \rightarrow Z$ be a contractive projection. The $ Q \circ P$ is a contractive projection 
from $X$ anto $F.$ The proof is complete.
\end{proof}
Now we present examples of Banach spaces $X$ in which any subspace being an existence set is a contractive set.
In  \cite{LTr}, the folowing result was shown. 
\begin{thm}
\label{ssmooth}
Let  $X$ be a Banach space and let $ Z
\subset X,$ be a linear subspace, which is an existence set.  Put
\begin{equation}
\label{smooth}
G_Z = \{ z \in Z \setminus \{ 0 \}: \hbox{ there exists exactly
one } f \in S(X^*) : f(z) = \| z \| \} .
\end{equation}
Assume that the norm closure of $ G_Z $ in $X$ is equal to $Z.$
Then there exists exactly one projection $ P \in {\cal P}(X,Z)$
such that $ \| P \| = 1, $ which means that $ Z$ is a contractive subset of $X.$
\end{thm}
\begin{rem}
\label{r5}
In particular, if $X$ is a smooth space, then (\ref{smooth}) is satisfied.
Applying Theorem \ref{ssmooth} it was shown in \cite{LTr}, that in $c_o$, $l_1$ and some Musielak-Orlicz sequence spaces any subspace $Z$ which is an existence set is one-complemented. Also it was shown in \cite{KL1} that any subspace $Z$ of 
the Lorentz sequence space 
$ l_{1,w} $ which is an existence set, is contractive.   
\end{rem}
\begin{cor}
 Let $X$ be a reflexive Banach space. Let $F \subset X$ be a bounded and convex existence set such that $ dim(Z=span(F)) < \infty .$ If $X$ is strictly convex and smooth then $X$ is contractive. If $X$ is finite-dimensional, strictly convex and  and smooth 
 Banach space then any bounded optimal set is contractive.  
\end{cor}
\begin{proof}
If $ dim(Z=span(F)) < \infty ,$ then $ int_Z(F) \neq \emptyset.$  By Theorem \ref{int1} and Theorem \ref{ssmooth} $F$ is contractive set. If $X$ is strictly convex and reflexive and $F$ is a bounded optimal set then by \cite{BM}, $F$ is a convex existence set.
If $X$ is finite-dimensional, then $ int_Z(F) \neq \emptyset$ for any $ F \subset X.$ The proof is complete. 
\end{proof}
The folowing result can be applied to sets satisfying $int_Z(F) = \emptyset.$ The proof of it is the same as in \cite{KL}.
\begin{thm} 
\label{cosun}
Let $ X$ be a reflexive smooth Banach space and let $F \subset X$ be a convex existence set such that $ F = cl(\bigcup_{t\in T}^{\infty} F_t), $ where $\{F_t\}_{t \in T}$ is a family of convex and compact existence sets ordered by $ \subset. $
Then $F$ is a contractive set. If $F$ is a convex, compact existence set then the assumption of reflexivity can be omitted. 
\end{thm}
\begin{proof}
Let $ x \in X$ and fix $ t \in T.$ By (\cite{KL},Th. 3.3) there exists $ P_tx \in R_{F_t}(x) $ such that for any $ d \in F_t$ and $s \geq 0$  
$$
\| Px - d\| \leq \| sx + (1-s)P_tx -d).
$$
Since $ X$ is smooth, by (\cite{BR2}, Lemma 1), $P_tx$ is uniquely determined. By (\cite{BR2}, Th. 1) $ F_t$ is a contractive set. (Up to now the reflexivity is not needed.) By (\cite{BR1}), Lemma 4, $ F$ is a contractive set. 
\end{proof}
\begin{rem}
\label{problem}
In general, we do not know under which conditions on a Banach space $X$ any existence and convex set can be represented as   $ F = cl(\bigcup_{t\in T}^{\infty} F_t), $ where $\{F_t\}_{t \in T}$ is a family of convex and compact existence sets ordered 
by $ \subset. $
Such a result has been proven \cite{KL}, Lemma 3.7 for reflexive K\"othe sequence spaces. For $l_p$ spaces such a result has been demonstrated in \cite{DE}. 
\end{rem}
\begin{rem}
\label{problem1}
Notice that in general a contractive set need not to be convex. Let, for example $ X = l_{\infty}^{(2)}.$ Define
$$
F = \{(x,y) \in X: |y|\leq |x| \}.
$$
Then after elementary, but tedious calculations, one can see that the mapping $P:X \rightarrow F$ defined by 
$$
P(x,y) = \begin{pmatrix} (x, sgn(y)x) & \text{ for $ x \geq 0, |y|> |x|$}\\
(x, -sgn(y)x) & \text{ for $ x < 0, |y|> |x|$} \\
(x,y) & \text{ for $|y|\leq |x|$} 
\end{pmatrix}
$$
is a contractive projection from $X$ onto $F.$ It would be be interesting to obtain a relation between existence and contractive sets in nonconvex case.  
\end{rem}
Notice that the following two well-known lemmas permit us to adopt the results proved for real Banach spaces in the case of complex Banach spaces. We present the proofs of them for a conveniece of the reader.    
\begin{lem}
\label{complex1}
Let $X$ be a complex Banach space with a Hamel basis $ H = \{h_t\}_{t \in T}.$ Let  $ X_{R}$ be a real linear space spanned by $ \{ h_t , ih_s\}_{s,t \in T}.$ Then $ H_R= \{ h_t , ih_s\}_{s,t \in T}$ is a Hamel basis  of $ X_{R}$ over $ \mathbb{R}.$ 
Let us equipp $X_{R}$ with a norm induced from $X,$ i.e.,
$$
\| \sum_{j=1}^n a_j h_{t_j} +\sum_{j=1}^n b_jih_{t_j} \| = \| \sum_{j=1}^n (a_j +ib_j) h_{t_j} \|
$$
Then the mapping 
$Iz : X \rightarrow X_{R}$ defined by 
$$ 
Iz (\sum_{j=1}^n a_j h_{t_j}) = \sum_{j=1}^n re(a_j)h_{t_j} +\sum_{j=1}^n im(a_j)ih_{t_j}
$$ 
is a linear surjective isometry over $\mathbb{R}$ which means that 
$ \| Iz(x) \| = \| x\|, $ $ Iz(x+y) = Iz(x) + Iz(y),$ $ Iz(ax) = aIz(x)$ for $a \in \mathbb{R}, $ and $ x,y \in X.$  In particular, $ X_{R}$ is a real Banach space. 
\end{lem}
\begin{proof}
Since $ H$ is a Hamel basis over $ \mathbb{C}, $  the set  $ \{ h_t , ih_s\}_{s,t \in T}$ is linearly independent over $ \mathbb{R}.$ 
The fact that $ Iz$ is a linear isometry over $\mathbb{R}$ follows immediately from the defintion of the norm in $ X_{R}.$ To prove surjectivity, fix $ x \in X_{R}.$ Then $ x = \sum_{j=1}^n a_j y_{t_j} +\sum_{k=1}^m b_k iy_{s_k}$ with 
$ a_j \neq 0$ and $ b_k \neq 0.$ 
Let $ S = \{ t_1,...,t_n, s_1,..,s_m\}.$ Let $ z = \sum_{j \in S} c_j y_j, $ where $c_j = a_j + ib_j.$ (We put $ a_j =0 $ if $ j \notin \{s_1,...,s_m \}$ and $ b_j =0 $ if $ j \notin \{t_1,...,t_n \}.$ It is clear that $ Iz(z) =x,$ which completes the proof. 
\end{proof}
\begin{lem}
\label{complex2}
Let $ X$ be a complex Banach space. Then $X$ is reflexive if and only if $ X_{R} $ is reflexive, $X$ is strictly convex if and only if $ X_{R}$ is strictly convex, $X$ is smooth if and only if $ X_{R} $ is smooth. Moreover $X$ satisfies (CFPP) if and only if 
$X_R$ satisfies (CFPP).
\end{lem}
\begin{proof}
The proofs of all above properties follows from Lemma \ref{complex1}. Before presenting them, fix $ f \in X^*. $ Then $ f = re(f) + i (im(f)). $ Put $ g = re(f).$ Since $ f(ix) = if(x)$ for any $ x \in X, $ we easily get that 
$ f(x) = g(x) - ig(ix).$ Moreover, $ \|f \| = \|g\|.$ Indeed, it is immediate that $ \| g\| \leq \|f\|.$ To prove a converse, fix $ x \in X,$ $\|x \| =1. $ Then $ f(e^{it}x) = g(e^{it}x)$ for some $ t \in [0,2\pi].$ 
Hence 
$$
|f(x)| = |f(e^{it}x)| =|g(e^{it}x)| \leq \|g\| \|e^{it}x\| = \|g\|
$$
which shows our claim. Moreover, by the above reasoning if $g \in (X_{R})^*$ then the mapping $f:X \rightarrow \mathbb{C}$ given by $ f(x) = g(x) - ig(ix) $ belongs to $ X^*$ and $ \|f\| = \|g\|.$ 
Hence, in particular, any $ f \in S(X^*)$ attains its norm at some $x\in S(X)$ if and only if any $g \in (X_R)^*$ attains its norm in some $ z \in S(X_R).$ Consequently, by the James Theorem $X$ is reflexive if and only if $X_R$ is reflexive.
Analogously, for any $x \in S(X)$ there exists exactly one $f \in S(X^*) $ satisfying $ f(x) = \|x\|=1$ if and only if  for any $z \in S(X_R)$ there exists exactly one $g \in S((X_R)^*) $ satisfying $ g(z) = \|z\|=1.$
Hence $X$ is smooth if and only if $X_R$ is smooth. The equivalence of strict convexity for $X$ and $X_R$ and the equivalence of (CFPP) for $X$ and $X_R$ follows immediately from Lemma \ref{complex1}, so we omit the proofs. 
\end{proof}
\begin{thm}
\label{complex3}
Theorem \ref{main}, Theorem \ref{countable}, Theorem \ref{s} (By a half-space determined by $ f \in S(X^*)$ we understand the set $\{ x \in X: re(f)(x) \leq d\}.),$ Theorem \ref{l1}, Theorem \ref{linf}, Theorem \ref{suff}, Theorem \ref{s1},
Theorem \ref{int1}, Theorem \ref{ssmooth} and Theorem \ref{cosun} holds true for any complex Banach space $X$.
\end{thm}
\begin{proof}
Let $ X$ be a complex Banach space and let $ F \subset X.$ Then by Lemma \ref{complex1}, $F$ is a convex set if and only if $ Iz(F)$ is a convex set,  $F$ is an existence set in $X$ if and only if $ Iz(F)$ is a and existence set in $X_R$ and 
$F$ is a contractive subset of $X$ if and only if $ Iz(F)$ is a contractive subset of $X_R.$ By Lemma \ref{complex2}, we can adopt the proofs of the above mentionned theorems given in the real case to the complex case. 
\end{proof}

$\begin{array}{l}
\textnormal{\small Maciej CIESIELSKI}\\
\textnormal{\small Institute of Mathematics}\\
\textnormal{\small Pozna\'{n} University of Technology}\\ 
\textnormal{\small Piotrowo 3A, 60-965 Pozna\'{n}, Poland}\\ 
\textnormal{\small email: maciej.ciesielski@put.poznan.pl;}\\
\textnormal{ }
\end{array}$

$\begin{array}{l}
\textnormal{\small Grzegorz LEWICKI}\\
\textnormal{\small Department of Mathematics and Computer Science}\\ \textnormal{\small Jagiellonian University}\\
\textnormal{\small 30-348 Krak\'ow, \L ojasiewicza 6, Poland}\\
\textnormal{\small email: grzegorz.lewicki@im.uj.edu.pl;}
\end{array}$

\end{document}